\documentclass{article}
\usepackage{amssymb}
\usepackage{amsmath}
\usepackage{amsthm}
\usepackage{mathabx}
\usepackage{enumerate}
\usepackage{xcolor}

\usepackage{hyperref}

\theoremstyle{plain}

\newtheorem{Lem}{Lemma}
\newtheorem{Cor}{Corollary}

\newtheorem{Thm}{Theorem}
\newcommand*{\supp}{\ensuremath{\mathrm{supp\,}}}

\newcommand*{\Ann}{\ensuremath{\mathrm{Ann\,}}}
\newcommand*{\R}{\ensuremath{\mathbb{R}}}

\newcommand*{\Z}{\ensuremath{\mathbb{Z}}}
\newcommand*{\C}{\ensuremath{\mathbb{C}}}

\begin{document}
	
	\date{}
	
	\author{
		L\'aszl\'o Sz\'ekelyhidi\\
		{\small\it Institute of Mathematics, University of Debrecen,}\\
		{\small\rm e-mail: \tt szekely@science.unideb.hu,}
	}
	
	\title{Spectral Synthesis on Direct Products}

	\maketitle
	
	\begin{abstract}
		In a former paper we introduced the concept of localization of ideals in the Fourier algebra of a locally compact Abelian group. It turns out that localizability of a closed ideal in the Fourier algebra is equivalent to the synthesizability of the annihilator of that closed ideal which corresponds to this ideal in the measure algebra. This equivalence provides an effective tool to prove synthesizability of varieties on locally compact Abelian groups. In another paper we used this method to show that when investigating synthesizability of a variety, roughly speaking, compact elements of the group can be neglected. Using these results, in this paper we completely characterize those locally compact Abelian groups on which spectral synthesis holds.
	\end{abstract}

	\footnotetext[1]{The research was supported by the  the
	Hungarian National Foundation for Scientific Research (OTKA),
	Grant No.\ K-134191.}\footnotetext[2]{Keywords and phrases:
	variety, spectral synthesis}\footnotetext[3]{AMS (2000) Subject Classification: 43A45, 22D99}
	
	\section{Introduction}
	Let $G$ be a locally compact Abelian group. Spectral synthesis deals with uniformly closed translation invariant linear spaces of continuous complex valued functions on $G$. Such a space is called a {\it variety}.  We say that a variety is {\it synthesizable}, if its finite dimensional subvarieties span a dense subspace in the variety. If every subvariety of a variety is synthesizable, then we say that {\it spectral synthesis} holds for the variety. If every variety on a topological Abelian group is synthesizable, then we say that {\it spectral synthesis holds} on the group. On commutative topological groups finite dimensional varieties of continuous functions are completely characterized: they consist of exponential polynomials. {\it Exponential polynomials} on a topological Abelian group are defined as the elements of the complex algebra of continuous complex valued functions generated by all continuous homomorphisms into the multiplicative group of nonzero complex numbers ({\it exponentials}), and all continuous homomorphisms into the additive group of all complex numbers ({\it additive functions}).  An {\it exponential monomial} is a function of the form
	$$
	x\mapsto P\big(a_1(x),a_2(x),\dots,a_n(x)\big)m(x),
	$$
	where $P$ is a complex polynomial in $n$ variables, the $a_i$'s are additive functions, and $m$ is an exponential. Every exponential polynomial is a linear combination of exponential monomials. For more about spectral analysis and synthesis on groups see \cite{MR2680008,MR3185617}.
	\vskip.2cm
	
	In \cite{MR2340978}, the authors characterized those discrete Abelian groups having spectral synthesis: spectral synthesis holds on the discrete Abelian group  if and only if it has finite torsion-free rank. In particular, from this result it follows, that if spectral synthesis holds on $G$ and $H$, then it holds on $G\times H$. Unfortunately, such a result does not hold in the non-discrete case. Namely, by the fundamental result of L.~Schwartz \cite{MR0023948}, spectral synthesis holds on $\R$, but D.~I.~Gurevich showed in \cite{MR0390759} that spectral synthesis fails to hold on $\R\times\R$. In this paper we enlighten this mysterious situation by proving that if spectral synthesis holds on the locally compact Abelian group $G$, then it holds on $G\times \Z$ as well, however, by  \cite{MR0390759}, it does not hold with $\R$ instead of $\Z$. Using this result we characterize those locally compact Abelian groups having spectral synthesis. The idea is that starting with $G$, which has spectral synthesis, we can extend it by either $\Z$, or, due to our result in \cite{Sze23b}, by a compact Abelian group so that the resulting group has spectral synthesis as well. Starting with $\R$, and applying this process, by the basic structure theory of locally compact Abelian groups we can reach any locally compact Abelian group with spectral synthesis, in finitely many steps. Our main tool is the concept of localizability of ideals in the Fourier algebra (see \cite{Sze23a}). 
	
	We shall use the following simple result about spectral synthesis.
	
	\begin{Thm}\label{subhom}
		If spectral synthesis holds on a topological Abelian group, then it holds on every continuous homomorphic image of it.
	\end{Thm}

\begin{proof}
	Assume that $G,H$ are topological Abelian groups, spectral synthesis holds on $G$, and $\Phi:G\to H$ is a continuous homomorphism of $G$ onto $H$. Let $W$ be a variety on $H$, then 
	$$
	V=\{\varphi\circ \Phi:\, \varphi\enskip\text{is in $W$}\} 
	$$
	is a variety on $G$, as it is easy to check. Clearly, the function $a:H\to\C$ is an additive function on $H$ is and only if $a\circ \Phi$ is an additive function on $G$, and $m:H\to \C$ is an exponential on $H$ if and only if $m\circ \Phi$ is an exponential on $G$. Consequently, $\varphi:H\to\C$ is an exponential monomial on $H$ if and only if it is an exponential monomial on $G$. As the exponential monomials in $V$ spane a dense subspace in $V$, it follows that the exponential monomials span a dense subspace in $W$, which proves our theorem.
\end{proof}
	
\section{Localization}
\indent In our former paper \cite{Sze23a} we introduced the concept of localization of ideals in the Fourier algebra of a locally compact Abelian group. We recall this concept here.

Let $G$ be a locally compact Abelian group and let $\mathcal A(G)$ denote its Fourier algebra, that is, the algebra of all Fourier transforms of compactly supported complex Borel measures on $G$. This algebra is topologically isomorphic to the measure algebra $\mathcal M_c(G)$. For the sake of simplicity, if the annihilator $\Ann I$ of the closed ideal $I$ in $\mathcal M_c(G)$ is synthesizable, then we say that the corresponding closed ideal $\widehat{I}$ in $\mathcal A(G)$ is synthesizable. For each derivation $D$ on $\mathcal A(G)$, we introduce (see \cite{Sze23a}) the set $\widehat{I}_{D,m}$ as the set of all functions $\widehat{\mu}$ in $\mathcal A(G)$ for which
$$
D\widehat{\mu}(m)=\int \Delta_{x,y_1,y_2,\dots,y_k}*f_{D,m}(0)\widecheck{m}(x)\,d\mu(x)=\widehat{\mu}(m)=0
$$
holds for each $k=1,2,\dots$ and $y_1,y_2,\dots,y_k$ in $G$. Then $\widehat{I}_{D,m}$ is a closed ideal in $\mathcal A(G)$. For a family $\mathcal D$ of derivations we write
$$
\widehat{I}_{\mathcal{D},m}=\bigcap_{D\in \mathcal D} \widehat{I}_{D,m}.
$$
Clearly, $\widehat{I}_{\mathcal{D},m}$ is a closed ideal as well. In other words, $\widehat{I}_{\mathcal{D},m}$ is the ideal of those functions in $\mathcal A(D)$ which are annihilated at $m$ by all derivations in the family of derivations $\mathcal D$.

The dual concept is the following: given an ideal $\widehat{I}$ in $\mathcal A(G)$ and an exponential $m$, the set of all derivations on $\mathcal A(G)$ which annihilate $\widehat{I}$ at $m$ is denoted by $\mathcal{D}_{\widehat{I},m}$. The subset of $\mathcal{D}_{\widehat{I},m}$ consisting of all polynomial derivations is denoted by $\mathcal{P}_{\widehat{I},m}$. We have the basic inclusion
\begin{equation}\label{basic}
\widehat{I}\subseteq \bigcap_m \widehat{I}_{\mathcal{D}_{\widehat{I},m},m}\subseteq \bigcap_m \widehat{I}_{\mathcal{P}_{\widehat{I},m},m}.
\end{equation}
We note that if $m$ is not a root of the ideal $\widehat{I}$, then $\mathcal{D}_{\widehat{I},m}=\mathcal{P}_{\widehat{I},m}=\{0\}$, consequently $\widehat{I}_{\mathcal{D}_{\widehat{I},m},m}=\widehat{I}_{\mathcal{P}_{\widehat{I},m},m}=\mathcal A(G)$, hence those terms have no affect on the intersection. 

The ideal $\widehat{I}$ is called {\it localizable}, if we have equalities in \eqref{basic}. The main result in \cite{Sze23a} is that $\widehat{I}$ is synthesizable if and only if it is localizable. We shall use this result in the subsequent paragraphs.

\section{The Fourier algebra of $G\times \Z$}

It is known that every exponential $e:\Z\to \C$ has the form
$$
e(k)=\lambda^k
$$
for $k$ in $\Z$,where $\lambda$ is a nonzero complex number, which is uniquely determined by $e$. For this exponential we use the notation $e_{\lambda}$. It follows that for every commutative topological group $G$, the exponentials on $G\times \Z$ have the form $m\otimes e_{\lambda}:(g,k)\mapsto m(g)e_{\lambda}(k)$, where $m$ is an exponential on $G$, and $\lambda$ is a nonzero complex number. Hence the Fourier transforms in $\mathcal A(G\times \Z)$ can be thought as two variable functions defined on the pairs $(m,\lambda)$, where $m$ is an exponential on $G$, and $\lambda$ is a nonzero complex number.
\vskip.2cm


Let $G$ be a locally compact Abelian group. For each measure $\mu$ in $\mathcal M_c(G\times \Z)$ and for every $k$ in $\Z$ we let: 
$$
S_k(\mu)=\{g:\, g\in G\enskip\text{and}\enskip (g,k)\in \supp \mu\}.
$$
As $\mu$ is compactly supported, there are only finitely many $k$'s in $\Z$ such that $S_k(\mu)$ is nonempty. We have
$$
\supp \mu=\bigcup_{k\in \Z} (S_k(\mu)\times \{k\}),
$$
and
$$
S_k(\mu)\times \{k\}=(G\times \{k\})\cap \supp \mu.
$$
It follows that the sets $S_k(\mu)\times \{k\}$ are pairwise disjoint compact sets in $G\times \Z$, and they are nonempty for finitely many $k$'s only. The restriction of $\mu$ to $S_k(\mu)\times \{k\}$ is denoted by $\mu_k$. Then, by definition
$$
\langle \mu_k, f\rangle=\int f\cdot \chi_k\,d\mu
$$
for each $f$ in $\mathcal C(G\times \Z)$, where $\chi_k$ denotes the characteristic function of the set $S_k(\mu)\times \{k\}$. In other words,
$$
\int f\,d\mu_k=\int f(g,k)\,d\mu(g,l)
$$
holds for each $k$ in $\Z$ and for every $f$ in $\mathcal C(G\times \Z)$. Clearly, $\mu= \sum_{k\in \Z} \mu_k$, and this sum is finite. 

\begin{Lem}\label{convi}
	Let $\mu$ be in $\mathcal M_c(G\times \Z)$. Then, for each $k$ in $\Z$, we have 
	$$
	\mu_k=\mu_0*\delta_{(0,k)}.
	$$
\end{Lem}

Here $\delta_{(0,k)}$ denotes the Dirac measure at the point $(0,k)$ in $G\times \Z$.

\begin{proof}
We have for each $f$ in $\mathcal C(G\times \Z)$:
$$
\langle \mu_0*\delta_{(0,k)},f\rangle= \int \int f(g+h,l+n)\,d\mu_0(g,l)\,d\delta_{(0,k)}(h,n)=
$$
$$
\int f(g,l+k)\,d\mu_0(g,l)=\int f(g,k)\,d\mu(g,l)=\langle \mu_k,f\rangle.
$$
\end{proof}

\begin{Lem}\label{suppconv}
	Let $(\widehat{\mu}_i)$ be a sequence in $\mathcal A(G\times \Z)$ which converges to $\widehat{\mu}$ in $\mathcal A(G\times \Z)$. Then there exists an $i_0$ such that for each $i\geq i_0$ we have $S_k(\mu_i)\subseteq S_k(\mu)$ for each $k$ in $\Z$. 
\end{Lem}

\begin{proof}
	Assuming the contrary, we may suppose that, for each $i$ there exists a $k_i$ in $\Z$ such that $k_i$ is in the projection of the support of $\mu_i$ onto $\Z$, but $k_i$ is not in the projection of the support of $\mu$ onto $\Z$. Then we can define a continuous function $f$ on $G\times \Z$ such that $\int f\,d\mu_i=1$ for each $i$, and $\int f\,d\mu=0$, which contradicts to $\widehat{\mu}_i\to \widehat{\mu}$. 
\end{proof}

Given a closed ideal $\widehat{I}$ in $\mathcal A(G\times \Z)$ for each $\widehat{\mu}$ in $\widehat{I}$ we define the measure $\mu_G$ in $\mathcal M_c(G)$ by
$$
\langle \mu_G,\varphi\rangle=\int \varphi(g)\,d\mu(g,l),
$$
whenever $\varphi$ is in $\mathcal C(G)$. Clearly, every $\varphi$ is in $\mathcal C(G)$ can be considered as a function in $\mathcal C(G\times \Z)$, hence this definition makes sense, further we have
$$
\langle \mu_G,\varphi\rangle=\int \varphi(g)\,d\mu_0(g,l).
$$

\begin{Lem}\label{closedid}
	If $I$ is a closed ideal in $\mathcal M_c(G\times \Z)$, then the set $I_G$ of all measures $\mu_G$ with $\mu$ in $I$, is a closed ideal in $\mathcal M_c(G)$.
\end{Lem}

\begin{proof}
Clearly $\mu_G+\nu_G=(\mu+\nu)_G$ and $\lambda\cdot \mu_G=(\lambda\cdot \mu)_G$. Let $\mu_G$ be in $I$ and $\xi$ in $\mathcal M_c(G)$. Then we have for each $\varphi$ in $\mathcal C(G)$:
$$
\langle \xi*\mu_G,\varphi\rangle=\int \int \varphi(g+h)\,d\xi(g)\,d\mu_G(h)=\int \int  \varphi(g+h)\,d\xi(g)\,d\mu(h,l).
$$
On the other hand, we extend $\xi$ from $\mathcal M_c(G)$ to $\mathcal M_c(G\times \Z)$ by the definition
$$
\langle \tilde{\xi},f\rangle=\int f(g,0)\,d\xi(g)
$$
whenever $f$ is in $\mathcal C(G\times \Z)$. Then 
$$
\langle \tilde{\xi}_{G},\varphi\rangle=\int \varphi(g)\,d\tilde{\xi}_0(g,l)=\int \varphi(g)\,d\xi(g)=\langle \xi,\varphi\rangle,
$$
that is $\tilde{\xi}_{G}=\xi$. Finally, a simple calculation shows that
$$
\langle \xi*\mu_G,\varphi\rangle=\langle (\tilde{\xi}*\mu)_G,\varphi\rangle,
$$
hence $\xi*\mu_G=(\tilde{\xi}*\mu)_G$ is in $I_G$, as $\tilde{\xi}*\mu$ is in $I$.
\vskip.2cm

Now we show that the ideal $I_G$ is closed. Assume that $(\mu_i)$ is a generalized sequence in $I$ such that the generalized sequence $(\mu_{i,G})$ converges to $\xi$ in $\mathcal M_c(G)$. This means that 
$$
\lim_i \int \varphi(g)\,d\mu_{i,G}(g)=\int \varphi(g)\,d\xi(g)
$$
holds for each $\varphi$ in $\mathcal C(G)$. In particular, for each exponential $m$ on $G$ we have 
$$
\lim_i \int \widecheck{m}(g)\,d\mu_{i,0}(g,l)=\lim_i \int \widecheck{m}(g)\,d\mu_{i,G}(g)=\int \widecheck{m}(g)\,d\xi(g)=\int \widecheck{m}(g)\,d\tilde{\xi}_0(g,l).
$$
In other words,
$$
\lim_i \widehat{\mu}_{i,0}=\widehat{\tilde{\xi}}_0
$$
holds. It follows
$$
\lim_i \mu_{i,0}=\tilde{\xi}_0,
$$
consequently
$$
\tilde{\xi}_k=\tilde{\xi}_0*\delta_{(0,k)}=\lim_i \mu_{i,0}*\delta_{(0,k)}=\lim_i \mu_{i,k}.
$$
By Lemma\ref{suppconv}, we have $S_k(\mu_{i,0})\subseteq S_k(\tilde{\xi}_0)$ for large $i$. We may assume that  $S_k(\mu_{i,0})\subseteq S_k(\tilde{\xi}_0)$ holds for each $i$. Then we infer
$$
\tilde{\xi}=\sum_k \tilde{\xi}_k=\sum_k \lim_i \mu_{i,k}=\lim_i \sum_k  \mu_{i,k}=\lim_i \mu_{i},
$$
where we used the fact that in each sum all those terms are zero, where $S_k(\tilde{\xi}_0)$ is empty. As $I$ is closed, hence $\tilde{\xi}$ is in $I$, which proves that $\tilde{\xi}=\tilde{\xi}_{G}$ is in $I_G$, that is, $I_G$ is closed.
\end{proof}

For the Fourier transforms of the measures $\mu_k$ we have:
$$
\widehat{\mu}_k(m,\lambda)=\int \widecheck{m}(g) \lambda^{-l}\,d\mu_k(g,l)=\lambda^{-k} \cdot \int \widecheck{m}(g)\,d\mu(g,l),
$$
in particular
$$
\widehat{\mu}_0(m,\lambda)=\int \widecheck{m}(g)\,d\mu(g,l)
$$
for each exponential $m$ on $G$ and for every nonzero complex number $\lambda$.
We shall use the following simple observation.
	Let $\widehat{I}$ be a closed ideal in the Fourier algebra $\mathcal A(G\times \Z)$ and let $\widehat{\mu}$ be in $\widehat{I}$. If $D$ is a polynomial derivation  with the generating polynomial $f_{D,m,\lambda}$ at $(m,\lambda)$ (see \cite{Sze23a}) which annihilates $\widehat{I}$ at the exponential $(m,\lambda)$, then $D\widehat{\mu}_k(m,\lambda)=0$ holds if and only if 
	$$
	\int f_{D,m,\lambda}(g,k)\widecheck{m}(g)\,d\mu(g,l)=0.
	$$
	Indeed, we have
	$$
	D\widehat{\mu}_k(m,\lambda)=\int f_{D,m,\lambda}(g,l)\widecheck{m}(g)\,d\mu_k(g,l)=
	\int f_{D,m,\lambda}(g,k)\widecheck{m}(g)\,d\mu(g,l).
	$$

Our main result in this paper is the following.

\begin{Thm}\label{main}
	Let $G$ be a locally compact Abelian group. Then spectral synthesis holds on $G$ if and only if it holds on $G\times \Z$.
\end{Thm}

\begin{proof}
	If spectral synthesis holds on $G\times Z$, then it obviously holds on its continuous homomorphic images, in particular, it holds on $G$, which is the projection of $G\times \Z$ onto the first component..
	\vskip.2cm
	
	Conversely,  we assume that spectral synthesis holds on $G$. This means, that every closed ideal in the Fourier algebra of $G$ is localizable, and we need to show the same for all closed ideals of the Fourier algebra of $G\times \Z$. 
	\vskip.2cm
	
	We consider the closed ideal $\widehat{I}$ in the Fourier algebra $\mathcal A(G\times \Z)$, and we assume that $\widehat{I}$ is non-localizable, that is, there is a measure $\nu$ in $\mathcal M_c(G\times \Z)$ such that $\widehat{\nu}$ is annihilated by $\mathcal P_{\widehat{I},m,\lambda}$ for each $m$ and $\lambda$, but $\widehat{\nu}$ is not in $\widehat{I}$. We show that $\widehat{\nu}_G$ is in $\widehat{I}_G$; then it will follow that $\widehat{\nu}$ is in $\widehat{I}$, a contradiction. 
	\vskip.2cm
	
	Suppose that a polynomial derivation $d$ annihilates $\widehat{I}_G$ at $m$. Then we have
	$$
	d\widehat{\mu}_G(m)=\int p_{d,m}(g)\widecheck{m}(g)\,d\mu_G(g)=\int p_{d,m}(g)\widecheck{m}(g)\,d\mu(g,l)=0
	$$
	for each $\widehat{\mu}$ in $\widehat{I}_G$ and for every exponential $m$ on $G$, where $p_{d,m}:G\to\C$ is the generating polynomial of $d$ at $m$.
	Then we define the polynomial derivation $D$ on the Fourier algebra $\mathcal A(G\times \Z)$ by
	$$
	D\widehat{\mu}(m,\lambda)=\int p_{d,m}(g) \widecheck{m}(g)\lambda^{-l}\,d\mu(g,l).
	$$
If $\widehat{\mu}$ is in $\widehat{I}$, then we have
$$
D\widehat{\mu}_k(m,\lambda)=\int p_{d,m}(g) \widecheck{m}(g)\lambda^{-l}\,d\mu_k(g,l)=\int p_{d,m}(g) \widecheck{m}(g)\,d\mu(g,l)\cdot \lambda^{-k}=0
$$
for each $k$ in $\Z$. As $\widehat{\mu}=\sum_{k\in \Z} \widehat{\mu}_k$, it follows that $D\widehat{\mu}(m,\lambda)=0$ for each $\widehat{\mu}$ in $\widehat{I}$. In other words, $D$ is in $\mathcal P_{\widehat{I},m,\lambda}$ for each exponential $m$ and nonzero complex number $\lambda$. In particular, $\widehat{\nu}$ is annihilated by $D$:
$$
D\widehat{\nu}(m,\lambda)=\int p_{d,m}\widecheck{m}(g)\lambda^{-l}\,d\nu(g,l)=0.
$$
It follows
$$
d\widehat{\nu}_G(m)=D\widehat{\nu}_{0}(m,\lambda)=\int p_{d,m}(g)\widecheck{m}(g)\,d\nu(g,l)=0.
$$
As $d$ is an arbitrary polynomial derivation which annihilates $\widehat{I}_G$ at $m$, we have that $\widehat{\nu}_G$ is annihilated by $\mathcal P_{\widehat{I}_G,m}$ for each $m$. As spectral synthesis holds on $G$, the ideal $\widehat{I}_G$ is localizable, consequently $\widehat{\nu}_G$ is in $\widehat{I}_G$, which implies that $\widehat{\nu}$ is in $\widehat{I}$, and our theorem is proved.
\end{proof}

\begin{Cor}
	Let $G$ be a compactly generated locally compact Abelian group. Then spectral synthesis holds on $G$ if and only if $G$ is topologically isomorphic to $\R^a\times \Z^b\times F$, where $a\leq 1$ and $b$ are nonnegative integers, and $F$ is an arbitrary compact Abelian group.
\end{Cor}

\begin{proof}
	By the Structure Theorem of compactly generated locally compact Abelian groups (see \cite[(9.8) Theorem]{MR0156915}) $G$ is topologically isomorphic to
	$$
	\R^a\times \Z^b\times F,
	$$
	where $a,b$ are nonnegative integers, and $F$ is a compact Abelian group. If spectral synthesis holds on $G$, then it holds on its projection $\R^a$. By the results in \cite{MR0390759}, spectral synthesis holds on $\R^a$ if and only if $a\leq 1$, hence $G$ is topologically isomorphic to $\R^a\times \Z^b\times F$ where $a\leq 1$ and $b$ are nonnegative integers, and $F$ is a compact Abelian group.  
	\vskip.2cm
	
	Conversely, let $G=\R\times \Z^b\times F$ with $b$ is a nonnegative integer, and $F$ a compact Abelian group. By the fundamental result of L.~Schwartz in \cite{MR0023948}, spectral synthesis holds on $\R$. By repeated application of Theorem \ref*{main}, we have that spectral synthesis holds on $\R\times \Z^b$ with any nonnegative integer $b$. Finally, by the results in \cite{Sze23b}, in particular, by Theorem 1., spectral synthesis holds on $\R\times \Z^b\times F$. Our proof is complete.
\end{proof}

\begin{Cor}
	Let $G$ be a locally compact Abelian group.  Let $B$ denote the closed subgroup of all compact elements in $G$. Then spectral synthesis holds on $G$ if and only if $G/B$ is topologically isomorphic to $\R^n\times F$, where $n\leq 1$ is a nonnegative integer, and $F$ is a discrete torsion free Abelian group of finite rank.
\end{Cor}

\begin{proof}
First we prove the necessity. If spectral synthesis holds on $G$, then it holds on $G/B$. By \cite[(24.34) Theorem]{MR0156915}, $G/B$ has sufficiently enough real characters. By \cite[(24.35) Corollary]{MR0156915}, $G/B$ is topologically isomorphic to $\R^n\times F$, where $n$ is a nonnegative integer, and $F$ is a discrete torsion-free Abelian group. As spectral synthesis holds on $\R^n\times F$, it holds on the continuous projections $\R^n$ and $F$. Then we have $n\leq 1$, and the torsion-free rank of $F$ is finite, by \cite{MR2340978}. 
\vskip.2cm

For the sufficiency, if $F$ is a torsion-free discrete Abelian group with finite rank, then it is the (continuous) homomorphic image of $\Z^k$ with some nonnegative integer $k$. By repeated application of Theorem \ref{main},  we have that spectral synthesis holds on $\R\times \Z^k$, and then it holds on its continuous homomorphic image $\R\times F$. Finally, by \cite[Theorem 1]{Sze23b}, we have that spectral synthesis holds on $G$. 
\end{proof}

	\end{document}